\documentclass[a4paper]{amsart} 
\usepackage{amsmath,amsthm,amssymb,amsfonts,mathrsfs,mathtools,color,hyperref, mathtools,crop, graphicx, enumitem}
\usepackage[top=3.5cm, bottom=3cm, left=2.5cm, right = 3.5cm, marginparwidth = 2cm]{geometry}
\usepackage{todonotes}
\usepackage{color}

\theoremstyle{plain}
\begingroup

\newtheorem{theorem}{Theorem}
\newtheorem*{theorem*}{Theorem}
\newtheorem{corollary}[theorem]{Corollary}

\newtheorem{lemma}[theorem]{Lemma}
\endgroup

\theoremstyle{definition}
\begingroup

\endgroup

\theoremstyle{remark}
\begingroup
\newtheorem*{remark}{Remark}

\endgroup 

\numberwithin{equation}{section}
\newcommand{\N}{\mathbb N} 
 
\newcommand{\R}{\mathbb R} 
\newcommand{\M}{{\mathscr M}}

\newcommand{\wto}{\rightharpoonup}

\newcommand{\E}{{\mathcal E}}
\newcommand{\W}{{\mathcal W}}

\renewcommand{\H}{{\mathcal H}}
\newcommand{\spt}{\operatorname{spt}}

\newcommand{\cc}{\Subset}

\newcommand{\eps}{\varepsilon}

\newcommand{\dx}{\,\mathrm{d}x}
\renewcommand{\d}{\,\mathrm{d}}
 
 \newcommand{\ds}{\,\mathrm{dA}}
 \newcommand{\B}{\mathcal B}

\begin{document}

\title[Helfrich's Energy and Constrained Minimisation]{Helfrich's Energy and Constrained Minimisation}

\author{Stephan Wojtowytsch}
\address{Stephan Wojtowytsch\\Department of Mathematical Sciences\\Durham University\\Durham DH1\,1PT, United Kingdom}
\email{s.j.wojtowytsch@durham.ac.uk}

\date{\today}


\begin{abstract}
For every $g\in\mathbb{N}_0$ and $\epsilon>0$, we construct a smooth genus $g$ surface embedded into the unit ball with area $8\pi$ and Willmore energy smaller than $8\pi + \epsilon$. From this we deduce that a minimising sequence for Willmore's energy in the class of genus $g$ surfaces embedded in the unit ball with area $8\pi$ converges to a doubly covered sphere for all $g\in\N_0$.  We obtain the same result for certain Canham-Helfrich energies with $\chi_K\leq 0$ without genus constraint and show that Canham-Helfrich energies with $\chi_K>0$ are not bounded from below in the class of smooth surfaces with area $S$ embedded into a domain $\Omega\Subset \R^3$.

Furthermore, we prove that the class of connected surfaces embedded in a domain $\Omega\Subset\R^3$ with uniformly bounded Willmore energy and area is compact under varifold convergence.  
\end{abstract}

\maketitle

\section{Introduction and Main Results}

Let $M$ be a closed $C^2$-surface embedded in $\R^3$. Then the Canham-Helfrich energy of $M$ is defined by
\[
\E(M) =  \int_M \chi_H\,(H-H_0)^2 + \chi_K\, K\d\H^2
\]
where $H$ and $K$ are the mean and Gaussian curvatures of $M$ respectively and $\chi_H, \chi_K$ and $H_0$ are material parameters. This bending energy is commonly used in the modelling of biological membranes. The most commonly studied case corresponds to $\chi_H \equiv \frac14$ and $\chi_K, H_0 \equiv 0$ and is known as Willmore's energy
\[
\W(M) = \frac14\int_M|H|^2 \d\H^2.
\]
A well-known (non-compact) example of Gro\ss e-Brauckmann shows that if $H_0\neq 0$, $\E$ need not be lower semi-continuous under varifold convergence. Namely, in \cite{grosse1993new} a sequence of surfaces $M_k$ is constructed which converges to a multiplicity two plane in a measure sense and satisfies $H_{M_k}\equiv 1$ for all $k\in\N$. If $H_0=0$, the energy does not depend on the orientation of $M$.

The parameter $\chi_H$ must be positive to obtain an energy bound from below. Due to the Gauss-Bonnet theorem, if $\chi_K$ is constant the second term in the Helfrich functional is of topological nature as $\int_MK\d\H^2 = 4\pi(1-g)$ where $g\in\N_0$ is the genus of the surface $M$. Thus when Helfrich's energy is supposed to be minimised in a certain genus class, the second term is usually dropped. We investigate the influence of prescribed topological genus on the minimisation problem for Willmore's energy and of the parameter $\chi_K$ on Helfrich's energy without prescribed genus. 

Let $g\in \N_0$, $S>0$ and $\Omega\subset\R^3$ open. Denote by $\M_{g,S,\Omega}$ the space of closed connected orientable genus $g$ surfaces which are $C^2$-embedded in $\Omega$ with surface area $S$ and by $\M_{S,\Omega}$ the union of all $\M_{g,S,\Omega}$ over $g\in\N_0$. Our main result is the following.

\begin{theorem}\label{theorem approximation}
Let $m\in\N$, $m\geq 2$, $g\in \N_0$ and $\eps>0$. Then there exists $M\in \M_{g,4\pi m, B_1(0)}$ such that
\[
\W(M) < 4\pi m + \eps.
\]
\end{theorem}

For the proof, we show that we can connect two concentric spheres with almost equal radii by a large number of catenoids. This does not change the area or Willmore's energy much since catenoids are minimal surfaces, but changes the topology to arbitrary genus. A further perturbation with small Willmore energy adds a sufficient amount of area. The argument is similar to \cite{Muller:2013vz}, where two spheres were connected by one catenoid. Our construction is more analytic than geometric and allows for any finite number of catenoids, whereas the construction of \cite{Muller:2013vz} requires (almost) a whole hemisphere per catenoid.

This has important implications for curvature energies.

\begin{corollary}\label{theorem willmore}
Let $g\in \N_0, m\in\N, m\geq 2$ and consider $\Omega = B_1(0)$, $S= 4m\pi$. Then every sequence $M_k\in \M_{g, S,\Omega}$ such that
\[
\W(M_k)\to \inf\left\{\W(M)\:|\:M\in \M_{g, S,\Omega}\right\} = 4\pi m
\]
converges to an $m$-fold covered unit sphere as varifolds, independently of $g$.
\end{corollary}

Convergence holds in the sense of varifolds and in particular as Radon measures on $\R^3$. This result differs from the unconstrained case \cite{bauer:2003er} or minimisation among $C^2$-boundaries with prescribed isoperimetric ratio \cite{MR3176354}. In both cases, there exists a smooth embedded (i.e.\ multiplicity $1$) surface of genus $g$ which minimises Willmore's energy among all surfaces of genus $g$ (which bound a domain with certain isoperimetric ratio, in the second case). 

\begin{corollary}\label{theorem negative}
Denote by $\E$ Helfrich's energy with constant parameters $\chi_K< 0<\chi_H$ and $H_0=0$. Let $m\in\N, m\geq 2$ and specify $\Omega = B_1(0)$, $S=4m\pi$. Then every sequence $M_k\in \M_{S,\Omega}$ such that
\[
\E(M_k)\to \inf\left\{\E(M)\:|\:M\in \M_{S,\Omega}\right\} = 4\pi\,(4\chi_Hm - \chi_K)
\]
converges to a higher multiplicity unit sphere $\mu = m\cdot \H^2|_{S^3}$ as varifolds and we have $\E(\mu) < \liminf_{k\to\infty}\E(M_k)$. If $M\in \M_{S', B_1(0)}$ for some $S'>0$ and 
\[
\E(M) \leq  4\chi_HS'
\]
holds, then $M$ is a topological sphere. Furthermore, if $\chi_K<-4\chi_H$, then for any open $\Omega\cc\R^3$, $S>0$, $C>0$ the functional $\E$ is bounded from below in the class of smooth manifolds $\M_{S,\Omega}$ and in the varifold closure of
\[
\B_{S,\Omega,C}:= \{M\in \M_{S,\Omega}\:|\:\E(M) < C\},
\]
 but not in the union of the closures
\[
\bigcup_{k=1}^\infty \overline{\B_{S,\Omega, k}}.
\]
\end{corollary}

The theorem has some implications for the use of Helfrich's energy in the modelling of lipid bilayers. The multiple covering of a single sphere is unphysical since a biological membrane separating two domains is usually the location of chemical exchange. The higher multiplicity does not increase effective surface area; on the contrary, it would make the transport of any exchanged species more difficult.
 Obviously, the situation of the corollary is highly idealised, but probably similar phenomena could be observed under more generic conditions. 

We also suggest that it might be more appropriate to consider the lower semi-continuous envelope with respect to varifold convergence of Helfrich's energy in the class of $C^2$-boundaries than its direct extension to curvature varifolds (discussed below).

The case $\chi_K>0$ is entirely unphysical. Here we can consider non-constant material parameters. Assume that there are measurable functions $\chi_H, \chi_K$ and $H_0$ associated to each surface $M\in \M_{S,\Omega}$.

\begin{corollary}\label{theorem positive}
Let $\Omega\subset\R^3$ open and $r>0$ such that $\overline{B_r(x)}\subset\Omega$ for some $x\in \R^3, r>0$. Let $\E$ be Helfrich's energy with parameters $\chi_H, \chi_K$ and $H_0$ satisfying the bounds
\[
||\,\chi_H\,||_{L^\infty(M)}\leq C, \qquad ||\,H_0\,||_{L^2(M)}\leq C, \qquad \delta \leq \chi_K \leq C
\]
for some $C, \delta>0$ independent of $M\in \M_{S,\Omega}$. Assume that $\mu = 4\pi\,r^2\cdot \delta_x$ is a point measure or $\mu = \H^2|_{\partial B_r(x)}$. Then there exists a sequence $M_k\in \M_{4\pi\,r^2,\Omega}$ such that $\H^2|_{M_k}\stackrel*\wto \mu$ as Radon measures and $\E(M_k)\to -\infty$. In the second case, even varifold convergence holds.
\end{corollary}

Corollaries \ref{theorem willmore}, \ref{theorem negative} and \ref{theorem positive} easily follow from Theorem \ref{theorem approximation} and the reverse estimate given in Lemma \ref{lemma mueller roeger}. We remark that unlike genus, connectedness is stable in the minimisation problem.

\begin{theorem}\label{theorem compactness with connected support}
Let $K, M>0$ and $\Omega\cc\R^n$ open. The class of integral $2$-varifolds $V$ in $\R^n$ satisfying
\begin{enumerate}
\item 
\[
\spt(\mu_V)\subset \overline{\Omega},\qquad \mu_V(\overline{\Omega}) \leq M, \qquad \W(V) \leq K\quad\text{and}
\]
\item $\spt(\mu_V)$ is connected
\end{enumerate}
is (sequentially) compact under the convergence of varifolds. The same holds for the closure with respect to varifold convergence of connected manifolds which are $C^2$-embedded into $\Omega$ with surface area bounded by $M$ and Willmore energy bounded by $C$. 
\end{theorem}

This has a direct implication for minimising Willmore's energy in a suitable topological class.

\begin{corollary}\label{corollary minimisation}
Let $\Omega\cc\R^n$ be open and $S>0$. Then there exists a $2$-varifold $V$ with mass measure $\mu_V$ such that
\begin{enumerate}
\item $\spt(\mu_V)\subset\overline\Omega$ is connected, $\mu_V(\overline\Omega) = S$ and
\item $V$ minimises $\W$ among the varifolds satisfying (1).
\end{enumerate}
The same holds if we add the assumption that $V$ is a varifold limit of connected embedded $C^2$-surfaces with uniformly bounded Willmore energy and surface area $S$ in (1).
\end{corollary}

Corollary \ref{corollary minimisation} follows directly from the Theorem \ref{theorem compactness with connected support}, the definition of varifold convergence and the lower-semicontinuity of Willmore's energy.

\subsection{Varifolds}

Willmore's energy is defined for general varifolds $V$ (see e.g.\ \cite{Allard:1972vh} or \cite{simon1983lectures}) through the integral over the squared weak mean curvature
\[
\W(V) = \frac14\int|H|^2\d\mu_V
\]
if the first variation $\delta V$ is absolutely continuous with respect to the mass measure $\mu_V$ and the density $H$ of $\delta V$ with respect to $\mu_V$ (i.e.\ weak mean curvature) is in $L^2(\mu_V)$. We also extend Helfrich's energy to the class of Hutchinson's curvature varifolds (varifolds with square integrable second fundamental form $A$, see \cite{hutchinson19862nd}) as
\[
\E(V) = \chi_H\int |H|^2\d\mu_V + \chi_K \int K\d V
\]
where $2K = |H|^2 - |A|^2$. $|A|$ is the Frobenius norm of the second fundamental form. This definition gives the usual Gaussian curvature on $C^2$-manifolds. For the purposes of this article, we only need to work with varifolds that are constant integer density multiples of smooth surfaces. In this case, the Gaussian curvature of the varifold agrees with the usual Gaussian curvature of the underlying manifold since the density cancels out in the defining equation of $A$. Thus our results only depend on the fact that $M$ and an integer multiple of $M$ have the same Gauss curvature, or even more generally on the fact that
\[
\int K \d V = \theta\,\int_M K \d\H^2
\]
if $V$ is the integral varifold given by $M$ and the constant density $\theta\in \N$. This assumption is sensible, since the total curvature cannot distinguish between the immersed manifolds $N_x:= M\cup (M+x)$ for different $x\in\R^3$ which are the union of $M$ and a translate of itself. The multiplicity two case can be thought of as $x= 0$.

Recall the following result about varifolds supported in the unit ball.

\begin{lemma}\cite[Theorem 1]{Muller:2013vz}\label{lemma mueller roeger}
Let $V$ be an integral 2-varifold with square integrable mean curvature $H$ such that $\spt(\mu_V)\subset \overline{B_1(0)}$. Then we have
\[
\W(V) \geq \mu_V(B)
\]
and equality holds if and only if $\mu = k\cdot\H^2|_{S^2}$ for an integer $k\in\N$.
\end{lemma}

Finally, recall the following localised Li-Yau inequality originally due to L.~Simon. A proof formulated for manifolds can be found in \cite[Lemma 1]{MR1650335}, but the same argument applies to integral varifolds.

\begin{lemma}\label{lemma li-yau}
Let $V$ be an integral varifold with $H\in L^2(\mu_V)$ and $r>0$. Then
\begin{equation}\label{eq li-yau}
\Theta^{2}(x) := \limsup_{s\to 0}\frac{\mu_V(B_s(x))}{\pi s^2} \leq \frac{\mu_V(B_r)}{\pi r^2} + \frac1{4\pi}\int_{B_r}|H|^2\d\mu_V
\end{equation}
\end{lemma}

\section{Proofs}

\subsection{Approximation of Spheres}

In this section, we will prove Theorem \ref{theorem approximation}. First we flatten the unit sphere slightly to have a flat segment on which we can easily glue two surfaces together. We denote by $D_r = B_r(0)$ the disc of radius $r$ around the origin in $\R^2$.

\begin{lemma}[``flattening a sphere"]\label{lemma flattening sphere}
Let $\eps>0$. Then there exists $\delta_0>0$ such that for every $0<\delta<\delta_0$ there exists a convex closed $C^\infty$-sphere $M_\eps\subset B_1(0)$ in $\R^3$ such that 
\begin{align*}
M_\eps\cap \left[ D_{2\delta}\times (0,1)\right]  &= \{x_3= 1-3\delta\} \cap \left[D_{2\delta}\times (0,1)\right],\\
M_\eps \setminus \left[D_{4\delta}\times (0,1)\right] &= S^2\setminus \left[D_{4\delta}\times (0,1)\right] 
\end{align*}
and
\[
\W(M_\eps) < 4\pi + \eps.
\]
\end{lemma}

\begin{proof}
Take $f\in C^\infty(-1,1)$, $f(t) = \sqrt{1-t^2}$ and
\[
f_\delta: B_1(0)\to\R, \qquad f_\delta(x) = f\circ r_\delta(|x|) = \sqrt{1 - r_\delta^2(|x|)\,}
\]
where $r_\delta\in C^\infty[0,1]$ satisfies 
\[
r_\delta(t) = \begin{cases} 3\delta & t\leq 2\delta\\ t &t\geq 4\delta\end{cases}, \qquad 0\leq r_\delta' \leq 1, \qquad 0\leq r_\delta'' \leq \frac4{\delta}.
\]
Then
\begin{align}
\partial_i f_\delta(x) &=(f'\circ r_\delta)\,r_\delta'\,\frac{x_i}{|x|}\\
\partial_{ij}^2 f_\delta(x) &= (f''\circ r_\delta)\,(r_\delta')^2\,\frac{x_ix_j}{|x|^2} + (f'\circ r_\delta)\, r_\delta'' \,\frac{x_ix_j}{|x|^2}  + (f'\circ r_\delta)\, r_\delta'\,\left[\frac{\delta_{ij}}{|x|} - \frac{x_ix_j}{|x|^3} \right].\label{eq second der}
\end{align}
It is easy to see that $D^2f_\delta$ is negative semi-definite since all three terms in the sum in \eqref{eq second der} are negative semi-definite, so $f_\delta$ is concave. Thus 
\[
M^\delta := \{x\in S^2\:|\:x_3 \leq 0\} \cup \{(x,f(x))\:|\:x\in D_1\}
\]
is a convex sphere. The topological type can also be found through the Gaussian curvature integral which coincides for $f$ and $f_\delta$ since their boundary values agree (Gauss-Bonnet Theorem). 

When we denote $f_0(x) = \sqrt{1-|x|^2\,}$, we observe that $|f_\delta - f_0|\leq 3\delta$ and 
\begin{align*}
|\partial_if_\delta - \partial_i f_0| &= \big[(f'\circ r_\delta)\,r_\delta' - (f'\circ r_0) \big]\,\frac{x_i}{|x|}\\
|\partial^2_{ij} f_\delta - \partial^2_{ij}f_0| &= \big[(f''\circ r_\delta)\,(r_\delta')^2 - (f''\circ r_0|)\big]\,\frac{x_ix_j}{|x|^2} +  (f'\circ r_\delta) \,r_\delta'' \,\frac{x_ix_j}{|x|^2}\\
	&\qquad\qquad - \big[(f'\circ r_\delta)  r_\delta' - (f'\circ r_0)\big]\,\left[\frac{x_ix_j}{|x|^3}  -  \,\frac{\delta_{ij}}{|x|}\right]
\end{align*}
The first term is small since $x_i/|x|$ is bounded and $f'(0) = 0$, so we can choose $\delta$ small enough to make $f_\delta$ and $f_0$ close in $C^1(\overline{B_1(0)})$. Curvature prevents us from making them $C^2$-close, but they are clearly $W^{2,2}$-close since
\begin{align*}
\left|\left|\big[(f''\circ r_\delta)\,(r_\delta')^2 - (f''\circ r_0)\big]\,\frac{x_ix_j}{|x|^2}\right|\right|_{L^2} &\leq 2\,||f''||_{L^\infty(-4\delta,4\delta)}\sqrt{\pi\,(4\delta)^2}  \\
\left|\left|(f'\circ r_\delta) \,r_\delta'' \,\frac{x_ix_j}{|x|^2} \right|\right|_{L^2}&\leq ||f'||_{L^\infty(-4\delta, 4\delta)}\, \left(\int_{B_{4\delta}(0)}(4/\delta)^2\dx\right)^{1/2}\\
\left|\left| \big[(f'\circ r_\delta)  r_\delta' - (f'\circ r_0)\big]\,\left[\frac{x_ix_j}{|x|^3}  -  \,\frac{\delta_{ij}}{|x|}\right] \right|\right|_{L^2} &\leq 2 \left(\int_{B_{4\delta}(0)} \left(\frac{2}{|x|}\right)^2\cdot2\, \left[\,||f''||_{L^\infty(-4\delta,4\delta)}|x|\,\right]^2\dx\right)^{1/2}
\end{align*}
all become small linearly with $\delta$. Since mean curvature $H_f$, volume element $\ds_f$ and Willmore integrand $w_f$ of the graph 
\[
\Gamma_f = \{(x,f(x))\:|\:x\in B_1(0) \subset \R^2\}
\]
of $f$ are given by
\begin{align*}
H_f &= \frac{(1+f_y^2)\,f_{xx} - 2\,f_xf_y\,f_{xy} + (1+f_x^2)\,f_{yy}}{(1+f_x^2+f_y^2)^{3/2}},\\
\ds_f &= \sqrt{1+f_x^2 + f_y^2\,}\quad\text{ and}
\end{align*}
$w_f = H_f^2\ds_f$, we see that $||w_f - w_{g}||_{L^1}$ is small if $|f-g|_{C^1}$ and $||f-g||_{W^{2,2}}$ are both small for some $g\in W^{2,2}(D_1)$. So we can chose $\delta$ small enough to make this as small as we need for $g= f_\delta$.
\end{proof}
 
\begin{remark}
The radial symmetry of the sphere simplifies the calculations above, but in fact any $C^2$-surface can be locally flattened around a point when written as a graph over its tangent space. This might be useful for a more general argument when minimising varifolds have double points.
\end{remark}
  
Next we create the handles by which we will connect spheres.

\begin{lemma}[``flattening a catenoid"]\label{lemma flattening catenoid}
Let $R\gg 1$. Then there exists a connected orientable $C^\infty$-manifold $\Sigma \subset\R^3$ such that
\begin{align*}
&\Sigma \setminus Z_R
\quad = \quad \big(\{x_3 = R+1/2\} \cup \{x_3 = -(R+1/2)\}\big)\setminus \overline{Z_R}
\end{align*}
where $Z_R$ is the cylinder $Z_R = D_{\cosh(R+1)}\times(-R+1/2, R+1/2)$ and furthermore
\[
\W(\Sigma) = O(e^{-2R}), \qquad \int_\Sigma K\d\H^2 = -4\pi
\]
where $K$ denotes the Gaussian curvature of $\Sigma$.
\end{lemma}

\begin{proof}
Define the surface of revolution
\[
\Sigma = \left\{\left.\begin{pmatrix}f(t)\,\cos\phi\\ f(t)\,\sin\phi\\ g(t)\end{pmatrix}\:\right|\:t, \phi\in\R\right\}.
\]
If $f(t) = \cosh(t)$ and $g(t) = t$, $\Sigma$ is the usual catenoid. We consider $f=\cosh$ and an even $C^\infty$-function $g$ satisfying
\[
g(t) = \begin{cases}t & |t|\leq R\\ R+1/2 & R\geq R+1\end{cases}, \qquad 0< g'(t)\leq 1 \text{ for }|t|< R+1, \qquad -4\leq g''(t)\leq 0 \text{ for }t\geq 0.
\]
Then clearly $\Sigma$ is connected as the continuous image of a connected set and given as the union of two planes outside the cylinder $Z_R$. The volume element $\ds$ and the mean curvature of $\Sigma$ are
\begin{align*}
\ds &= f\sqrt{(g')^2+(f')^2},\\
H &= \frac{ff''g' - ff'g'' - g'(f')^2-(g')^3}{f\,\left[(f')^2 + (g')^2\right]^{3/2}}\\
	&= \frac{g'(ff'' - (f')^2 - 1) +g'(1-(g')^2) - ff'\,g''}{f\,\left[(f')^2 + (g')^2\right]^{3/2}}\\
	&= \frac{g'(1-(g')^2) - ff'\,g''}{f\,\left[(f')^2 + (g')^2\right]^{3/2}}
\end{align*}
since $ff'' - (f')^2 - 1 = 0$ for $f=\cosh$. Thus
\begin{align*}
\W(\Sigma) &= 2\pi\int_0^\infty \frac{\left[g'\left(1-(g')^2\right) - ff'\,g''\right]^2}{f\,\left[(f')^2 + (g')^2\right]^{5/2}}\d t\\
	&\leq 4\pi\int_R^{R+1} \frac{(g')^2(1-(g')^2)^2}{f\,\left[(f')^2 + (g')^2\right]^{5/2}} + \frac{f\,(f')^2 (g'')^2}{\left[(f')^2 + (g')^2\right]^{5/2}}\d t\\
	&\leq 4\pi \int_R^{R+1} \frac{1}{f\,(f')^5}\d t + 4\pi \int_R^{R+1}\frac{f\,(g'')^2}{|f'|^3}\d t\\
	& = O(e^{-2R}).
\end{align*}
It remains to show that the total Gaussian curvature is $-4\pi$. When we orient $\Sigma$ by choice of the normal vector
\[
\nu = \frac{1}{f\,\sqrt{(f')^2 + (g')^2\,}} \begin{pmatrix} -f\,g'\,\cos\phi\\ - f\,g'\,\sin\phi\\ f\,f'\end{pmatrix}
\]
we see that every unit vector $\nu = (\sin\theta) e_\phi + (\cos\theta) e_z \neq (0,0,\pm1)$ is the normal $\nu_x$ at the unique point $x\in \Sigma$ determined by the $\phi$-coordinate and $t$ given by
\[
\tan\theta = - \frac{g'(t)}{f'(t)}.
\]
This is uniquely solvable except for $\tan\theta = 0$ by construction of $g$. We know that
\[
K = \frac{-(g')^2f'' + f'g'g''}{f\,\left[(f')^2 + (g')^2\right]^2} \leq 0\qquad (\text{since }f''\geq 0, f'g''\leq 0)
\]
is the determinant of the Gauss map $G:\Sigma\to S^2$, $G(x) = \nu_x$, so
\[
4\pi = \H^2(S^2) = \H^2(G(\Sigma)) = \int_\Sigma|K|\d\H^2 = - \int_\Sigma K\d\H^2.
\]
\end{proof}

Now we are ready to prove this section's main statement.

\begin{proof}[Proof of Theorem \ref{theorem approximation}]
We first give the proof for $m=2$. Let $\beta>0$ to be chosen later depending on $\eps, \delta>0$. Take $M_\beta$ constructed like in Lemma \ref{lemma flattening sphere}, $\delta>0$ such that $M_\beta$ coincides with the plane $\{x_3 = 1-3\delta\}$ inside the cylinder $D_{2\delta}\times(0,1)$. We may specify $\delta$ to be taken sufficiently small later. Take $0<\rho<\delta/2$ such that there are $g+1$ points $x_1,\dots, x_{g+1}$ in $D_{\delta/2}$ such that the discs $D_\rho(x_i)$ are pairwise disjoint.
 
Choose $R >0$ and $\Sigma$ like in Lemma \ref{lemma flattening catenoid} such that $\W(\Sigma) = O(e^{-2R})<\beta$. Then choose $\eta>0$ such that $\eta\cosh(R+1)<\rho$ and $\eta R< \delta^3$. Finally, define
\[
r = \frac{1- 3\delta -(2R+1)\eta}{1-3\delta}<1, \qquad \tilde M = M_\beta \cup r\cdot M_\beta.
\] 
Since $M_\beta$ is convex, this is  a smooth embedded manifold. By construction, inside the cylinders
\[
Z_i := D_\rho(x_i) \times \{0 < z < 1\}, \qquad i = 1,\dots, g+1
\]
$\tilde M$ is given by the union of the planes $\{z= 1-3\delta\}$ and $\{z= 1-3\delta - (2R+1)\eta\}$ which have separation $(2R+1)\eta$. Since the $Z_i$ are disjoint, we can replace $\tilde M$ inside each cylinder by
\[
x_i + \frac{1+r}2 \,e_z + \eta\cdot \big(\Sigma \cap Z_R\big).
\]
We call the resulting manifold $M$. It is clear that $M$ is a connected surface. Since both the total curvature integral and Willmore's energy are invariant under spacial rescaling and since $M$ is flat on the remaining segments, we have
\begin{align*}
\W(M) &= 2\,\W(M_\beta) + (g+1)\,\W(\Sigma) \hspace{12mm}< (g+3)\,\beta\\
\int_MK\d\H^2 &= 2\,\int_{M_\beta} K\d\H^2 + (g+1) \int_{\Sigma} K \d\H^2
	= 2\cdot 4\pi + (g+1)\cdot(-4\pi)
	= 4\pi\,(1-g)
\end{align*}
so that $M$ is a closed smoothly embedded orientable genus $g$ surface with small Willmore energy. Unfortunately, $M\not \subset B_1(0)$ and it is not clear after our modifications whether $\H^2(M) = 8\pi$. If $\H^2(M)>8\pi$, we only need to choose $\beta = \eps/(g+3)$ and set
\begin{equation}\label{rescaling to fit}
M_g = \sqrt{\frac{8\pi}{\H^2(M)}}\cdot M.
\end{equation}
The more complicated case is $\H^2(M)\leq 8\pi$. Then at least
\[
\H^2(M) \geq \left( 1 + r^2\right) \H^2\left(S^2 \setminus \left[D_{4\delta}\times (0,1)\right]\right) \geq 8\pi - C\,\delta^2
\]
since $\eta R<\delta^3$ and thus $r\geq 1- \delta^2$. Now consider only the inner sphere, which is still spherical around its south pole. Take a function $h\in C_c^\infty(D_r)$ on a small disc such that $h\geq 0$ and $h\not\equiv 0$. Then we may replace a neighbourhood of the south pole of the inner sphere by
\[
\tilde \Sigma^t = \left\{\left(x, -\sqrt{r^2 - |x|^2} + t\,h\left(\frac{x}{\alpha \sqrt t}\right)\right)\:\bigg|\:x\in B_r(0)\right\}.
\]
The resulting surface is denoted by $M^t$. Again, this does not change the topological type, but it changes the area and the Willmore functional by
\[
\H^2(M^t)\geq \H^2(M) + c\,t^2, \qquad \W(M_t) \leq \W(M) + C\,t
\]
as is computed in the proof of \cite[Proposition 2]{Muller:2013vz}, at least for suitable spherically symmetric $h$. Thus we can take $t = O(\delta)$ such that $\H^2(M^t)>8\pi$ and define $M_g$ again by \eqref{rescaling to fit}, this time choosing both $\beta$ {and} $\delta$ small enough depending on $\eps>0$.

In the case of $m\in \N$, $m\geq 3$, we simply consider $m$ concentric spheres and connect them by $m+ g-1$ catenoids. The modification at the south pole can always be done only for the innermost sphere. To picture that this procedure induces the correct topology, consider first connecting the outer spheres by $g+1$ catenoids. Then we connect the third sphere to the second by one catenoid. This, however, only blows up a small topological disc to a large one since the union of a catenoid and a sphere is homeomorphic to a sphere with a small disc around the north pole removed, i.e.\ a disc.
\end{proof}

\begin{remark}
If we fix a genus $g$, then we can even find a $C^2$-smooth map $f: (0,\eps)\times M_\eps \to \R^3$ which maps $(t,M_\eps)$ to $M_t$ constructed above. In particular, $f(t,M_\eps)$ is a $C^\infty$-smooth manifold for all $t\in (0,\eps)$. Clearly, the images converge as varifolds to an $m$-fold covered sphere as $t\searrow 0$. We can continue the evolution past the $m$-fold covered sphere in various ways. This describes a singularity in a geometric flow which may occur with decreasing Willmore energy in finite time. It is unclear whether such singularities may appear in the gradient flow of the Willmore functional.
\end{remark}

\subsection{Proofs of the Corollaries}
Let us use Theorem \ref{theorem approximation} to illustrate phenomena occurring when we minimise curvature energies under area constraint in the unit ball.

\begin{proof}[Proof of Corollary \ref{theorem willmore}]
By Theorem \ref{theorem approximation}, there exists a sequence $N_k\in \M:= \M_{g,4m\pi, B_1(0)}$ such that $\W(N_k) < 4m\pi + 1/k$. So 
\[
\inf\left\{\W(M) \:\big|\: M\in \M\right\} \leq 4m\pi.
\]
Now let $M_k$ be a minimising sequence in $\M$. Take a subsequence of $M_{k}$. Due to Allard's compactness theorem \cite{Allard:1972vh}, there exists an integral varifold $V$ with square integrable mean curvature $H$ such that a further subsequence converges to $V$ as varifolds and 
\[
\W(V) \leq \limsup_{k\to\infty}\W(M_k) = \inf\left\{\W(M) \:\big|\: M\in \M\right\} \leq 4m\pi.
\]
The convergence of varifolds implies the convergence of their mass measures as Radon measures, so $\mu_V(\overline{B_1(0)}) = 4m\pi$ and $\mu_V(\R^n\setminus\overline{B_1(0)}) = 0$, whence $\mu_V = m\cdot \H^{2}|_{S^2}$ by Lemma \ref{lemma mueller roeger}. Since every subsequence has a further subsequence which converges to the same limit and varifold convergence is topological (as a convergence of Radon measures), we see that the whole sequence converges.
\end{proof}

\begin{proof}[Proof of Corollary \ref{theorem negative}]
Since $\int_{M}|H|^2\d\H^2 > 4\H^2(M)$ by Lemma \ref{lemma mueller roeger} for manifolds in $B_1(0)$, a manifold $M$ satisfying
\[
4\chi_H\H^2(M) \geq \E(M) = \chi_H\int_{M}|H|^2\d\H^2 + 4\pi\,\chi_K\int_{M}K\d\H^2 \geq 4\chi_H\H^2(M) + 4\pi \chi_K\,(1-g) 
\]
has genus $g=0$. As before
\[
\inf \left\{\E(M)\:|\: M\in \M_{4\pi m, B_1(0)}\right\} = 16\pi m\,\chi_H -4\pi\,|\chi_K|
\]
is realised by smooth spheres converging to a multiplicity $m$ sphere. As noted before, a smooth multiplicity $m$-sphere $V: = m\cdot \H^2|_{S^2\otimes TS^2}$ has total Gaussian curvature $\int K\d V = 4\pi m$. Thus
\[
\E(V) = 4\pi m\,(4\chi_H - |\chi_K|) < 16\pi m\,\chi_H -4\pi\,|\chi_K| = \lim_{k\to\infty}\E(M_k).
\]
If $\chi_K<- 4\chi_H$, then multiplicity $m$-spheres illustrate that $\E$ is not bounded below on the varifold closure of smooth surfaces, since $m\cdot \H^2|_{S^2}$ can be approximated with finite energy $\E$.

Assume that $M_k$ is a sequence of smooth surfaces with energy $\E$ bounded by $C$ and $M_k$ converges to a varifold $V$. This implies that their genera and Willmore energies are bounded by
\[
g \leq \frac{C}{4\pi\,|\chi_K|}+1, \qquad \W(M_k) \leq \E(M_k) + 4\pi\,|\chi_K| 
\]
so 
\[
\int_{M_k}|A|^2 \d\H^2 = \int_{M_k}|H|^2 - 2K\d\H^2 \leq 4 \left[\,\E(M_k) + 4\pi\,|\chi_K|\,\right] + 8\pi\left[\frac{C}{4\pi\,|\chi_K|}+1\right].
\]
This is uniformly bounded in $k$, so $V$ is a curvature varifold \cite{hutchinson19862nd}. Clearly
\[
\E(V) \geq \chi_H \W(V) - |\chi_K| \,\W(V) \geq - |\chi_K| \int|A|^2\d V
\]
is a uniform bound from below in $\overline{\B_{S,\Omega,C}}$.
\end{proof}

The discontinuity is mathematically meaningful. As the catenoid collapses away, two spheres remain in the limit. The Gaussian integral does not see that these spheres happen to coincide.

\begin{proof}[Proof of Corollary \ref{theorem positive}]
To approximate a multiplicity one-sphere by manifolds $M_k$, insert a sphere of radius $1/k$ into a sphere of radius $\approx r$ and connect the two by $g_k$ catenoids, $g_k\to \infty$. Willmore's energy is close to $8\pi$, so the total energy is 
\begin{align*}
\E(M_k) &= \int_{M_k}\chi_{H}^k (H-H_0^k)^2 \d\H^2 + \int_{M_k} \chi_K^k\,K\d\H^2 \\ 
	&\leq C\,\int_{M_k}2\big(|H|^2 + |H_0^k|^2\big) \d\H^2 + \delta \int_{M_k\cap \{K<0\}} K\d\H^2 + C \int_{M_k\cap\{K>0\}} K\d\H^2\\
	&\leq 2 C\left( 8\pi+1 + C^2\right) + \delta \int_{M_k} K\d\H^2 + \frac{C}4 \int_{M_k\cap \{K>0\}} H^2\d\H^2\\
	&\leq 2C\left(8\pi + 1 +C^2\right) +4\pi \delta(1-g_k) + C^3/4
\end{align*}
since $K = \lambda_1\lambda_2 \leq (\lambda_1+\lambda_2)^2/4 = H^2/4$ if $\lambda_1$ and $\lambda_2$ have the same sign. Clearly, this goes to $-\infty$ as $g_k\to \infty$.
To approximate a Dirac measure, we approximate a multiplicity $m$-sphere of radius $r_m = r/\sqrt{m}$ with genus $g$-manifolds $\tilde M_m$, $g \gg m$.
\end{proof}

\subsection{Connectedness} We conclude this article with a more positive result.

\begin{proof}[Proof of Theorem \ref{theorem compactness with connected support}]
Since varifold convergence is weak* convergence and the pre-dual of varifolds is the space of continuous functions on a compact manifold $\overline\Omega \times G(3,2)$ (in particular, separable), the topology is locally metrisable. Therefore, compactness and sequential compactness coincide on the bounded set we consider.

Take a sequence of varifolds $V_k$ with mass measures $\mu_k$. By Allard's compactness theorem \cite{Allard:1972vh}, there is a subsequence converging to a limit varifold $V$ with mass measure $\mu$. Take a subsequence of $V_k$ (not relabelled) for which the supports $\spt(\mu_k)$ of the mass measures converge to a compact set $K\subset\overline{B_R(0)}$ in the Hausdorff distance. We will show that points which lie in the $K\setminus\spt(\mu_V)$ are atoms of size at least $4\pi$ of a finite measure, thus there can only be finitely many such points. If $\spt(\mu_k)$ is connected for all $k$, also the limit $K$ is connected and there cannot be any isolated points, so $K\subset\spt(\mu)$. The reverse inclusion always holds, so $\spt(\mu) = K$ is connected.
 
Assume that $x\in K \setminus \spt(\mu)$. Denote the weak* limit $\alpha = \lim_{k\to\infty} |H_k|^2\cdot\mu_k$ (for a subsequence along which it exists). We can take a sequence $x_k\in \spt(\mu_k)$, $x_k\to x$ with 
\[
\theta_k(x_k) = \Theta^2_{\mu_k}(x_k) = \limsup_{r\to0}\frac{\mu_k(B_r(x_k))}{\pi r^2} \geq 1
\]
since the Lebesgue points of $\theta_k$ lie dense. Take $\rho>0$ with $\mu(\overline{B_\rho(x)})=0$. Due to Lemma \ref{lemma li-yau} applied to $V_k$, we get
\[
1\leq \liminf_{k\to\infty} \frac{\mu_k(B_{\rho/2}(x_k))}{\pi (\rho/2)^2} + \frac1{4\pi}\int_{B_{\rho/2}(x_k)}|H_k|^2\d\mu_k \leq \liminf_{k\to\infty} \frac1{4\pi}\int_{B_{\rho}(x)}|H_k|^2\d\mu_k\leq \alpha\left(\overline{B_\rho(x)}\right)
\]
since the first term vanishes in the limit and $B_{\rho/2}(x_k)\subset B_\rho(x)$ for all large $k\in\N$. Taking $\rho\to 0$ shows that $\alpha(\{x\})\geq 4\pi$ and concludes the proof. Compare \cite[Lemma 3.5]{DW_conv} for a phase field version of this argument.

A diagonal sequence shows that the subclass of varifolds which arise as the weak* limits of embedded connected $C^2$-manifolds with suitable bounds is closed, hence it is compact as well.
\end{proof}

\section*{Acknowledgements}

I would like to thank P.W.~Dondl for his encouragement and for many stimulating discussions. I would also like to thank Durham University for financial support through a Durham Doctoral Studentship.

\bibliographystyle{alphaabbr}
\bibliography{full_corrected}

\end{document}